\documentclass[11pt]{article}

\usepackage{verbatim,latexsym,amsfonts,amsmath,amssymb,graphicx,fancyhdr,hyperref,xcolor, amsthm 
}
\usepackage{appendix,latexsym,amsfonts,amsmath,amssymb,graphicx,hyperref,amsthm,soul,verbatim,authblk,bbm}
\usepackage[framemethod=tikz]{mdframed}

\newcommand{\EE}{\mathbb{ E}}
\newcommand{\PP}{\mathbb{P}}

\newcommand{\R}{\mathbb{R}}
\newcommand{\C}{\mathbb{C}}

\newcommand{\HH}{\mathbb{H}}
\newcommand{\N}{\mathbb{N}}

\newcommand{\pa}{\partial}

\newcommand{\F}{{\cal F}}

\def\eps{\varepsilon}
\def\til{\widetilde}
\def\ha{\widehat}
\def\sem{\setminus}
\def\lin{\overline}
\def\ulin{\underline}

\def\ZZ{{\cal Z}}

\def\rotimes{{\otimes}}

\DeclareMathOperator{\dist}{dist} 
\DeclareMathOperator{\hcap}{hcap} 
\DeclareMathOperator{\Imm}{Im } \DeclareMathOperator{\Ree}{Re }

 \DeclareMathOperator{\cc}{c}
\DeclareMathOperator{\bb}{b}

\DeclareMathOperator{\LP}{LP}

\newcommand{\CLP}{\ha{\LP}}

\theoremstyle{plain}
\theoremstyle{plain}
\newtheorem{Theorem}{Theorem}[section]
\newtheorem{Lemma}[Theorem]{Lemma}
\newtheorem{Corollary}[Theorem]{Corollary}
\newtheorem{Proposition}[Theorem]{Proposition}

\theoremstyle{definition}
\newtheorem{Definition}[Theorem]{Definition}

\numberwithin{equation}{section}
\newcommand{\BGE}{\begin{equation}}
\newcommand{\BGEN}{\begin{equation*}}
\newcommand{\EDE}{\end{equation}}
\newcommand{\EDEN}{\end{equation*}}

\newcommand{\ind}{\mathbbm{1}}

\begin{document}
\title{Existence and Uniqueness of Nonsimple Multiple SLE}
\author{Dapeng Zhan
}
\affil{Michigan State University}
\date{\today}
\maketitle

\begin{abstract}
  We prove the existence and uniqueness of multiple SLE$_\kappa$ associated with any given link pattern for $\kappa\in (4,6]$. We also have the uniqueness for $\kappa\in (6,8)$. The multiple SLE$_\kappa$ law is constructed by first inductively constructing a $\sigma$-finite multiple SLE$_\kappa$ measure and then normalizing the measure whenever it is finite. The total mass of the measure satisfies the conformal covariance, asymptotics and PDE for multiple SLE$_\kappa$ partition functions in the literature subject to the assumption that it is smooth.
\end{abstract}

\section{Introduction}
The Schramm-Loewner evolution (SLE) is a one-parameter ($\kappa\in(0,\infty)$) family of random curves, which are known to be the scaling limits of various two-dimensional discrete models. A chordal SLE curve grows in a simply connected domain connecting two distinct boundary points (more precisely, prime ends). It arises naturally as the candidate of the scaling limit of some lattice model with Dobrushin boundary condition.

When the  Dobrushin boundary condition is replaced by alternative boundary condition, the scaling limit is expected to be multiple SLE. In a multiple $N$-SLE$_\kappa$ configuration, there are $N$ random curves connecting $2N$ distinct marked boundary points of a simply connected domain with the property that the conditional law of any curve given all other curves is a chordal SLE$_\kappa$ curve in one connected component of the complement of the other curves.


When $\kappa\in (0,4]$,  multiple  SLE$_\kappa$  associated with any given link pattern  exists (cf.\ \cite{Julien-Euler},\cite{KL-multiple},\cite{Wu-multiple}) and is unique in law (cf.\ \cite{BPW-multiple}). This is also true for $\kappa\in \{16/3,6\}$ (\cite{BPW-multiple}). When $N=2$, the existence and uniqueness of $2$-SLE$_\kappa$ are known for all $\kappa\in (0,8)$ (cf.\ \cite{HW-hSLE,MSW}).

In this paper we study multiple SLE$_\kappa$ for $\kappa\in (4,8)$. The main result is  the existence and uniqueness  for any $\kappa\in(4,6]$. We also prove the uniqueness in the case $\kappa\in(6,8)$. In fact, for all $\kappa\in (0,8)$, given any link pattern $\alpha$ in a domain $D$, we construct a $\sigma$-finite measure $Q^D_\alpha$ on the space of $N$-tuples of curves associated with $\alpha$ using a cascade relation. When $\kappa\le 6$, $Q^D_\alpha$ is finite, and its normalization is the law of multiple SLE associated with $\alpha$. For $\kappa\in(6,8)$, the same statement holds as long as $Q^D_\alpha$ is known to be finite.

We call the total mass $H^D_\alpha=|Q^D_\alpha|$ a multiple SLE$_\kappa$ partition function. We focus on the case $D=\HH$, in which case $H(\alpha):= H^{\HH}_\alpha$ is a function on an open subset of $\R^{2N}$. Using the construction of $Q^D_\alpha$, we prove that $H$ satisfies conformal covariance and asymptotics for multiple SLE pure partition functions (cf.\ \cite{KP-multiple}). 
We  construct a continuous local martingale using $H$.  If $H$ is known to be smooth, then It\^o's calculus shows that $H$ satisfies a system of PDEs for multiple SLE partition function introduced by Dub\'edat in \cite{Julien-Euler}.

\section*{Acknowledgement}
The author thanks Xin Sun for rasing the question about nonsimple multiple SLE, and  Pu Yu for introducing the irreducible Markov chain theory to the author.

\section{Preliminary}
\subsection{Measures and kernels}
 We use $f(\mu)$ (instead of $f_*(\mu)$) to denote the pushforward of a measure $\mu$ under a measurable map $f$. We use $g\cdot\mu$ or $g(\omega)\cdot \mu(d\omega)$ to denote a measure obtained by weighting a measure $\mu$ by a nonnegative function $g$.

Let $(\Omega_j,\F_j)$, $j=1,2$, be measurable spaces. A function $\nu:\Omega_1\times \F_2\to [0,\infty)$ is called a finite kernel from $\Omega_1$ to $\Omega_2$ if (i) for any $\omega_1\in\Omega_1$, $\nu(\omega_1,\cdot)$ is a finite measure on $(\Omega_2,\F_2)$; and (ii) for any $B\in\F_2$, $\nu(\cdot,B)$ is a measurable function on $(\Omega_1,\F_1)$.  Suppose $\mu$ is a $\sigma$-finite measure on $\Omega_1$, and $\nu$ is a finite kernel from $\Omega_1$ to $\Omega_2$. Using Carath\'eodory's extension theorem, we define $\mu\rotimes \nu$ to be the unique $\sigma$-finite measure  on $\Omega_1\times \Omega_2$, which satisfies that,
\BGE \mu\rotimes\nu(B_1\times B_2)=\int_{B_1} \mu(d\omega_1) \nu(\omega_1,B_2),\quad B_1\in\F_1,\quad B_2\in\F_2.\label{P1mu}\EDE
Note that $\mu\otimes\nu$ is well defined even if $\nu(\omega_1,\cdot)$ is not defined for $\omega_1$ in some $\mu$-null set.
Sometimes we write $\mu\rotimes \nu$ as $\mu(d\omega_1)\rotimes \nu(\omega_1,d\omega_2)$ to emphasize the dependence of $\nu$ on $\omega_1$.

We list a few simple facts.
\begin{itemize}
\item Let $P_1$ be the projection $(\omega_1,\omega_2)\mapsto \omega_1$. Then
\BGE P_1(\mu\rotimes \nu)=\nu(\cdot,\Omega_2)\cdot \mu\ll \mu. \label{P1mu*}\EDE 
\item If $f$ is a nonnegative measurable function on $\Omega_1$, then
\BGE f(\omega_1)\cdot \mu(d\omega_1) \rotimes \nu(\omega_1,d\omega_2)=\mu(d\omega_1) \rotimes f(\omega_1) \nu(\omega_1,d\omega_2).   \label{fdot}\EDE
\item Let $\mu$ be a $\sigma$-finite measure on $(\Omega_1,\F_1)$, and $\nu,\nu^\#$ be finite kernels from $(\Omega_1,\F_1)$ to $(\Omega_2,\F_2)$, which satisfy $ \nu(\omega_1,\cdot)=\nu(\omega_1,\Omega_2)\nu^\#(\omega_1,\cdot)$. Let $P_1$ be the projection  $(\omega_1,\omega_2)\mapsto \omega_1$. Then
   \BGE \mu\rotimes \nu=P_1(\mu\rotimes \nu)\rotimes \nu^\#.\label{P1linnu}\EDE
\item An $\Omega_1\times\Omega_2$-valued random element $(X,Y)$ satisfies that the conditional law of $Y$ given $X$ is $\nu(X,\cdot)$ iff the law $\mu$ of $(X,Y)$ satisfies
    \BGE \mu=P_1(\mu)(dx)\otimes \nu(x,dy).\label{conditional}\EDE
\end{itemize}

\subsection{Schramm-Loewner evolution}
We focus on chordal SLE, which is first defined using chordal Loewner equation. Below we provide a brief review (cf.\ \cite{Law1}). Let $W\in C([0,T);\R)$ for some $T\in (0,\infty]$. The chordal Loewner equation driven by $W$ is
\BGE \pa_t g_t(z)=\frac 2{g_t(z)-W_t},\quad g_0(z)=z.\label{Loewner}\EDE
For every $z\in \C$, there is $\tau(z)\in [0,T]$ such that $[0,\tau(z))$ is the biggest interval on which $t\mapsto g_t(z)$ is defined. Let $K_t=\{z\in\HH:\tau(z)\le t\}$, $0\le t<T$, where $\HH:=\{z\in\C:\Imm z>0\}$. The $g_t$ and $K_t$ are respectively called chordal Loewner maps and hulls driven by $W$. It turns out that, for each $t$, $g_t$ maps $\HH\sem K_t$ conformally onto $\HH$, and fixes $\infty$.

If we differentiate (\ref{Loewner}) w.r.t.\ $z$, then we get
\BGE \pa_t g_t'(z)=\frac {-2 g_t'(z)}{g_t(z)-W_t},\quad g_0'(z)=1.\label{Loewner'}\EDE

If the map $(t,z)\mapsto g_t^{-1}(z)$ extends continuously from $[0,T)\times \HH$ to $[0,T)\times \lin\HH$, then the continuous curve $\eta(t):=g_t^{-1}(W_t)$, $0\le t<T$, in $\lin\HH$, is called the chordal Loewner curve driven by $W$. Such $\eta$ may not exist in general. When it exists, then for any $t\in[0,T)$, $\HH\sem K_t$ is the unbounded component of $\HH\sem \eta[0,t]$.

Let $W_t=\sqrt\kappa B_t$, where $\kappa>0$ and $B_t$ is a standard Brownian motion. Then the chordal Loewner curve $\eta$ driven by $W$ is known to exist (\cite{RS}), and is called a chordal SLE$_\kappa$ in $\HH$ from $0$ to $\infty$. If $D$ is a simply connected domain with two distinct prime ends $a,b$, we may find a conformal map $f$ from $\HH$ onto $D$, which sends $0,\infty$ to $a,b$. Then $f\circ\eta$  modulo time parametrization (due to the non-uniqueness of $f$) is called a chordal SLE$_\kappa$ in $D$ from $a$ to $b$. Such $f\circ\eta$ lies in the prime end closure of $D$, and when $\pa D$ is locally connected, lies in $\lin D$.

An SLE$_\kappa$ curve is simple if $\kappa\in (0,4]$; space-filling if $\kappa\ge 8$; non-simple and non-space-filling if $\kappa\in (4,8)$. For the rest of the paper we assume $\kappa\in (0,8)$. The following proposition follows immediately from \cite[Lemma 3.6]{BPW-multiple}.

\begin{Proposition}
  Let $D'\subset D$ be two simply connected domains with locally connected boundary that share two distinct prime ends $a,b$.  Let $\PP$ be the law of  chordal SLE$_\kappa$ in $D$  from $a$ to $b$. Then $\PP[\eta\subset \lin{D'}]>0$.
  \label{Prop-abs}
\end{Proposition}

\subsection{Multiple SLE}
Let $D$ be a simply connected domain with locally connected boundary and $2N$ distinct prime ends $a_1,b_1,\dots,a_N,b_N$. An $N$-tuple $(\eta_1,\dots,\eta_N)$ of random curves in $\lin D$ is called a multiple $N$-SLE$_\kappa$ in $D$ associated with $\alpha:=((a_1,b_1),\dots,(a_N,b_N))$ if, for any fixed $r\in\{1,\dots,N\}$, the conditional law of $\eta_r$ given all $\eta_j$, $j\ne r$, is an SLE$_\kappa$ curve from $a_r$ to $b_r$ in the connected component of $D\sem \bigcup_{j\ne r}\eta_j$ which contains neighborhoods of $a_r,b_r$ in $D$. A necessary condition for the existence of these curves is that there exist $N$ mutually disjoint simple curves $\gamma_1,\dots,\gamma_N$ in $D$ such that $\gamma_j$ connects $a_j$ with $b_j$, in which case $\alpha$ is called a link pattern.

By conformal invariance of SLE, one may assume that $\pa D$ is analytic at the points in $\alpha$. In the case $\kappa\le 4$, the  $N$-SLE$_\kappa$ associated with any link pattern  is known to exist (cf.\ \cite{Julien-Euler},\cite{KL-multiple},\cite{Wu-multiple}). The important idea  in \cite{KL-multiple} is that it is useful to view SLE$_\kappa$ as a non-probability measure, which is the SLE$_\kappa$ law (probability measure) multiplied by its partition function.

Let $\mu^D_{(a,b)}$  denote the  chordal SLE$_\kappa$ law in $D$ from $a$ to $b$.
Assume that $\pa D$ is analytic at $a,b$. The boundary Poisson kernel in $D$ at $(a,b)$ is  defined by $$P^D_{(a,b)}=\lim_{\eps\to 0^+} \frac{\pi}{\eps^2} G_D(a+\eps {\bf n}_a,b+\eps{\bf n}_b),$$
where $G_D(\cdot,\cdot)$ is the Dirichlet Green's function for $D$, and ${\bf n}_a,{\bf n}_b$ are inward normals. 

Define two constants depending on $\kappa$: \BGE \bb=\frac{6-\kappa}{2\kappa},\qquad \cc=\frac{(6-\kappa)(3\kappa-8)}{2\kappa},\label{bc}\EDE
which are respectively called the boundary scaling exponent and central charge of SLE$_\kappa$.
The $H^D_{(a,b)}:=(P^D_{(a,b)})^{\bb} $ and $Q^D_{(a,b)}:=H^D_{(a,b)}\mu^D_{ (a,b)}$ are respectively called the SLE$_\kappa$ partition function and SLE$_\kappa$ measure  for $(D;a,b)$.
The partition function satisfies the conformal covariance: if $f$ is a conformal map on $D$ that is extended in neighborhoods of $a,b$, then
\BGE H^D_{(a,b)}=|f'(a)|^{\bb}|f'(b)|^{\bb} H^{f(D)}_{(f(a),f(b))}.\label{covariance}\EDE
When $D=\HH$, we have the special value $H^{\HH}_{(a,b)}=|a-b|^{-2\bb}$. If $\kappa\le 6$, then $\bb\ge 0$, and $H^D_{(a,b)}$ satisfies the monotonicity: if $D'\subset D$ contains neighborhoods of $a,b$ in $D$, then $H^{D'}_{(a,b)}\le H^D_{(a,b)}$. The direction of the inequality  is reversed if $\kappa>6$.

Let $\alpha=((a_1,b_1),\dots,(a_N,b_N))$ be a link pattern in $D$. Assume that $\pa D$ is analytic at points in $\alpha$. Let $\kappa\in(0,4]$. The $N$-SLE$_\kappa$ measure in $D$ with link pattern $\alpha$ is defined as a measure $Q^D_\alpha$ defined by
\BGE Q^D_\alpha(d\ulin\eta)= Q^D_\alpha(d(\eta_1,\dots,\eta_N))=Y(\ulin\eta)\cdot \prod_{j=1}^N Q^D_{(a_j,b_j)}(d\eta_j),\label{Q}\EDE
where
$$ Y(\ulin\eta) =I(\ulin\eta)\exp\Big(\frac {\cc}2 \sum_{j=2}^N m_D(K_j(\ulin\eta))\Big),$$ 
$I(\ulin\eta)$ is the indicator of the event that $\eta_1,\dots,\eta_N$ are mutually disjoint, $\cc$ is given by (\ref{bc}), $m_D$ is the Brownian loop measure in $D$ (cf.\ \cite{Brownian-loop}), and $K_j(\ulin\eta)$ is the collection of loops that intersect at least $j$ of the curves $\eta_1,\dots,\eta_N$. Such $Q^D_\alpha$ is a finite measure, and its normalization is an $N$-SLE$_\kappa$ law in $D$ with link pattern $\alpha$.

The uniqueness of multiple SLE$_\kappa$ for $\kappa\le 4$ is established in \cite{BPW-multiple}, which also proves the existence and uniqueness of  multiple SLE$_{16/3}$ using the convergence of FK Ising model. The same argument works for multiple SLE$_6$ as scaling limit of critical percolation. The work  also studies  multiple SLE$_\kappa$ for other $\kappa\in(4,6]$ and proves the existence and uniqueness based on a conjecture about the convergence of some random cluster model. When there are only $2$ curves,  the existence and uniqueness of $2$-SLE$_\kappa$ hold for all $\kappa\in(0,8)$ (cf.\ \cite{HW-hSLE,MSW}).

\section{Existence}
The main goal of the paper is to prove the existence and uniqueness of multiple SLE$_\kappa$ associated with any given link pattern for $\kappa\in(4,8)$, in which case the curves are nonsimple and have positive probability to intersect. We do not consider the case $\kappa\ge 8$ because in that case the existence is trivial and the uniqueness does not hold.

If one uses (\ref{Q})  to construct $Q^D_\alpha$  for $\kappa\in (4,8)$, then by \cite[Lemma 2.3]{BPW-multiple} its normalization is not the law of multiple SLE$_\kappa$ but the law of $(\eta_1,\dots,\eta_N)$ with the properties that, every $\eta_r$ conditional on all other $\eta_j$'s is an SLE$_\kappa$ curve from $a_r$ to $b_r$ in some complement domain of the given curves {\it conditioned to avoid these curves}. 

Fortunately, \cite[Proposition 3.2]{KL-multiple} provides an cascade relation of $Q^D_\alpha$, which could be used for the construction.
It says that the marginal measure on $(\eta_1,\dots, \eta_{N-1})$ in $Q^D_\alpha$ is absolutely continuous w.r.t.\ $Q^D_{\alpha'}$, $\alpha'$ being the $\alpha$ without the link $(a_N,b_N)$,  with Radon-Nikodym derivative $H^{D'}_{(a_N,b_N)}$, where $D'$ is the connected component of $D\sem \cup_{k=1}^{N-1}\eta_k$ whose boundary contains $a_N,b_N$. Since the  $\eta_N$ in $Q^U_\alpha$ conditionally on $(\eta_1,\dots, \eta_{N-1})$ has the law of $\mu^{D'}_{(a_N,b_N)}$, we may obtain $Q^D_\alpha$  by first weighting $Q^D_{\alpha'}$ by $H^{D'}_{(a_N,b_N)}$ and then sampling $\eta_N$ according to $\mu^{D'}_{(a_N,b_N)}$.

We will follow the above approach to construct $Q^D_\alpha$ for $\kappa\in (4,8)$ inductively on $N$.
It is useful to construct multiple SLE in possibly disconnected open sets.
Let $\cal U$ denote the collection of open subsets of $\C$.
For $U\in\cal U$ and $N\in\N$, let $\CLP_N^U$ denote the set of tuples $\alpha =((a_1,b_1),\dots,(a_N,b_N))$ of $2N$ distinct points on $\pa U$, such that each of the $2N$ points lies on the boundary of a unique connected component of $U$, which is simply connected  with locally connected boundary and whose boundary is  analytic at this point. We write $\alpha_j$ for the couple $(a_j,b_j)$ or the set $\{a_j,b_j\}$. If there exist $N$ mutually disjoint simple curves $\gamma_j$, $1\le j\le N$, in $U$, such that each $\gamma_j$ connects points of $\alpha_j$, then $\alpha$ is called a link pattern of $U$ of size $N$. Let $\LP^U_N$ denote the collection of all such $\alpha$. We are going to construct a measure $Q^U_\alpha$ for any $\alpha\in\CLP^U_N$ on the space $X^U_\alpha:=\prod_{j=1}^N X^U_{\alpha_j}$, where $X^U_{\alpha_j}$ is the space of unparametrized curves from $a_j$ to $b_j$.

If $\alpha=(a,b)\in \LP^U_1$, there exists a unique connected simply connected component $D$ of $U$, whose boundary contains and is analytic at $a,b$. Then we define $\mu^U_\alpha=\mu^D_\alpha$, $H^U_\alpha=H^D_\alpha$, and $Q^U_\alpha=Q^D_\alpha$. Let $\mu^U_\alpha =H^U_\alpha =Q^D_\alpha=0$ if $\alpha\in \CLP^U_1\sem \LP^U_1$. 

Using  the above convention and (\ref{fdot}), in the case that $U$ is a simply connected domain $D$, we may express the cascade relation of $Q^U_\alpha$ for $\kappa\in(0,4]$   symbolically  as
 \BGE Q^U_\alpha =Q^U_{P_{\ha N}(\alpha)}(d(\eta_1,\dots \eta_{N-1}))\rotimes Q^{U\sem \bigcup_{j=1}^{N-1}\eta_j}_{\alpha_N}(d\eta_N), \label{cascade}\EDE
where $P_{\ha N}(\alpha)$ is the $\alpha$ without the $N$-th component.




We introduce some symbols. Let $\N_N=\{j\in\N:j\le N\}$ for $N\in\N$, and $S_N$ be the symmetric group of $\N_N$. For a vector $\ulin x=(x_1,\dots,x_N)$ and $j\in\N_N$, we define
$$ P_{\ha j}(\ulin x)=(x_1,\dots,\ha x_j,\dots,x_N),\quad P_{\le j}(\ulin x)=(x_1,\dots,x_j),$$ 
     where $\ha x_j$ means that $x_j$ is missing in the list. 
For $\sigma\in S_N$, let
$$ \sigma(\ulin x)=(x_{\sigma^{-1}(1)},\dots,x_{\sigma^{-1}(N)}).$$ 
Note that $\sigma_1(\sigma_2(\ulin x))=(\sigma_1\sigma_2)(\ulin x)$ for any $\sigma_1,\sigma_2\in S_N$.
We will view these $P_{\ha j}$, $P_{\le j}$ and $\sigma$   as functions on  $\CLP^U_N$ or on $X^U_\alpha$ for some $\alpha\in \CLP^U_N$.

\begin{Definition}
For any $\kappa\in (0,8)$, $U\in {\cal U}$ and $\alpha\in \CLP^U_N$, $N\in\N$, the $Q^U_\alpha$ is a $\sigma$-finite measure on $X^U_\alpha$ defined inductively on $N$ using the induction basis: the definition of $Q^U_{(a,b)}$ when $(a,b)\in \CLP^U_1$ and the induction step: (\ref{cascade}).
\label{QUalpha}
\end{Definition}

We remark that (\ref{cascade}) makes sense because from $\kappa<8$ we know that $\eta_1,\dots,\eta_{N-1}$ following the law $Q^U_{P_{\ha N}(\alpha)}$ a.s.\ do not pass through points of $\alpha_N$, and so $\alpha_N\in\CLP^{U\sem  \bigcup_{j=1}^{N-1}\eta_j}_1$. Note that $U\sem  \bigcup_{j=1}^{N-1}\eta_j$ is not connected, and $\alpha_N$ may not belong to $\LP^{U\sem  \bigcup_{j=1}^{N-1}\eta_j}_1$ even if $\alpha\in \LP^U_N$.

The $N$-SLE$_\kappa$ law could also be described using the language of kernels. For $j\le N$, we use $\sigma_{[j,N]}$ to denote the cyclic permutation $(j,j+1,\dots,N)\in S_N$. Then
        \BGE P_{\ha N} =P_{\ha j}\circ  \sigma_{[j,N]}. \label{P-N-j}\EDE
In view of (\ref{conditional}), it is easy to see that $(\eta_1,\dots,\eta_N)$  is   an $N$-SLE$_\kappa$ in $U$ with link pattern $\alpha$ iff its  law $\mu$, a probability measure on $X^U_\alpha$, satisfies that
\BGE \mu=\sigma_{[j,N]}\Big(P_{\ha j}(\mu)(d(\eta_1,\dots,\ha\eta_j,\dots,\eta_N))\rotimes \mu^{U\sem\cup_{k\ne j}\eta_k}_{(a_j,b_j)}(d\eta_j)\Big),\quad j\in\N_N. \label{N-SLE-measure}\EDE

\begin{Definition}
  A $\sigma$-finite measure $\mu$ on $X^U_\alpha$ is called a multiple SLE$_\kappa$ measure in $U$ with link pattern $\alpha$ if it satisfies (\ref{N-SLE-measure}). Clearly, if such $\mu$ is finite and nontrivial, then its normalization $\mu/|\mu|$ is an $N$-SLE$_\kappa$ law in $U$ with link pattern $\alpha$.
\end{Definition}

In the following we are going to show that $Q^U_\alpha$ satisfies (\ref{N-SLE-measure}) and study its trivialness and finiteness.
Let $\alpha^k=P_{\le k}(\alpha)$, $1\le k\le N$. Then $\alpha=\alpha^N$, and
\BGE Q^U_{\alpha^k} =Q^U_{\alpha^{k-1}}(d(\eta_1,\dots \eta_{k-1}))\rotimes Q^{U\sem \bigcup_{j=1}^{k-1}\eta_j}_{\alpha_k}(d\eta_k),\quad 2\le k\le N. \label{cascade2}
\EDE


\begin{Lemma}
  $Q^U_\alpha$ is not trivial iff $\alpha\in \LP^U_N$. \label{Lem-trivial}
\end{Lemma}
\begin{proof}
  If $\alpha\in\CLP^U_N\sem\LP^U_N$, then there are $j<k\in\N_N$ such that $\alpha_j,\alpha_k$ lie on the boundary of the same component $D$ of $U$, and if $\eta_j$ is an SLE$_\kappa$ in $\lin D$ connecting $\alpha_j$, then   $\eta_j$ disconnects points of $\alpha_k$ in $D$. Thus, $\alpha_k\not\in \LP_1^{U\sem \bigcup_{s=1}^{k-1} \eta_s}$, which implies that $Q^U_{\alpha^{k-1}}$-a.s.\ $Q^{U\sem \bigcup_{j=1}^{k-1}\eta_j}_{\alpha_k}=0$. By (\ref{cascade2}) we get $0=Q^U_{\alpha^k}=Q^U_{\alpha^{k+1}}=\cdots= Q^U_{\alpha^N}=Q^U_\alpha$.

Suppose $\alpha \in \LP^U_N$.
 We call  $T$ a tube in $U$ connecting  $\alpha_j $ if it is the closure of a Jordan subdomain of one component $D_j$ of $U$, whose boundary is a union of the closure of two crosscuts of $D_j$ and two disjoint open boundary arcs, which respectively contain points of $\alpha_j$. Since $\alpha\in \LP^U_N$, there exist mutually disjoint tubes $T_j$ in $U$ connecting $\alpha_j$, $1\le j\le N$. Let ${\cal T}_j=\{\eta\in X^U_j: \eta\subset T_j\}$. By Proposition \ref{Prop-abs}, $Q^U_{\alpha^1}({\cal T}_1)>0$ and for any $2\le k\le N$, when $(\eta_1,\dots,\eta_{k-1})\in \prod_{j=1}^{k-1} {\cal T}_j$, $Q^{U\sem \bigcup_{j=1}^{k-1}\eta_j}_{\alpha_k}({\cal T}_k)>0$. Thus, by (\ref{cascade}) and induction, $Q^U_{\alpha^k}(\prod_{j=1}^k {\cal T}_j)>0$ for $1\le k\le N$, which means that  $Q^U_\alpha=Q^U_{\alpha^N}$ is not trivial.
\end{proof}

\begin{Lemma}
  If $\kappa\le 6$, then $Q^U_\alpha$ is finite. \label{Lem-finite}
\end{Lemma}
\begin{proof}
By (\ref{cascade}) and the monotonicity of $H^D_{(a,b)}$ when $\kappa\le 6$,
$$|Q^U_{\alpha^k}|\le \int Q^U_{\alpha^{k-1}}(d(\eta_1,\dots,\eta_{k-1})) |Q^U_{(a_k,b_k)}|=|Q^U_{\alpha^{k-1}}| H^U_{(a_k,b_k)},\quad 2\le k\le N,$$
which implies that $|Q^U_\alpha|\le \prod_{k=1}^N H^U_{(a_k,b_k)}<\infty$.
\end{proof}

The following lemma follows from Wu's work \cite{HW-hSLE}.

\begin{Lemma}
For any   $\alpha\in \CLP^U_2$, $Q^U_\alpha$ is finite, and
$ \sigma (Q^U_{\alpha}) =Q^U_{\sigma(\alpha)}$ for  $\sigma=(1,2)\in S_2$. 
\label{2-commute}
\end{Lemma}
\begin{proof}
If $\alpha\not\in\LP^U_2$, then $\sigma(\alpha)\not\in\LP^U_2$. By Lemma \ref{Lem-trivial}, $Q^U_\alpha=Q^U_{\sigma(\alpha)}=0$.
Suppose now $\alpha \in\LP^U_2$. For $j=1,2$, there is a simply connected component $D_j$ of $U$ whose boundary contains $\alpha_j$. Then $Q^U_{\alpha_1}=Q^{D_1}_{\alpha_1}$. If $D_1\ne D_2$, then for $Q^{D_1}_{\alpha_1}$-a.s.\ every $\eta_1$, $Q^{U\sem \eta_1}_{\alpha_2}=Q^{D_2}_{\alpha_2}$. Thus, by (\ref{cascade}), $Q^U_{\alpha}=Q^{D_1}_{\alpha_1}\times Q^{D_2}_{\alpha_2}$. Similarly, $Q^U_{\sigma (\alpha)}=Q^{D_2}_{\alpha_2}\times Q^{D_1}_{\alpha_1}$. So $ \sigma (Q^U_{\alpha}) =Q^U_{\sigma(\alpha)}$.

Now consider the case $D_1=D_2$. We may assume $U=D_1 $ since other components of $U$ are irrelevant.  By reversibility of chordal SLE$_\kappa$ and the fact that $H^U_{(a,b)}=H^U_{(b,a)}$, we may relabeling the points of $\alpha_j=\{a_j,b_j\}$, $j=1,2$, such that $a_1,b_1,b_2,a_2$ are ordered clockwise or counterclockwise on $\pa U$. Let $H^U_\alpha=|Q^U_\alpha|$.
By (\ref{cascade}) and (\ref{P1mu}),
\BGE H^U_\alpha=H^U_{\alpha_1}\int \mu^U_{\alpha_1}(d\eta_1) H^{U\sem \eta_1}_{\alpha_2}.\label{H1}\EDE  Let $f$ be a conformal map from $U$ onto $\HH$ such that $f(a_1)=0$ and $f(b_1)=\infty$. Let $x=f(a_2)$ and $y=f(b_2)$. By conformal covariance, \BGE \int \mu^U_{\alpha_1}(d\eta_1) H^{U\sem \eta_1}_{\alpha_2} =|f'(a_2)|^{\bb}|f'(b_2)|^{\bb} \int \mu^{\HH}_{(0,\infty)}(d\gamma) H^{\HH\sem \gamma}_{(x,y)}.\label{H2}\EDE
Now $r:=x/y\in (0,1)$ is the cross ratio of $\alpha$ in $U$.
By \cite[Lemma 3.5 and Proposition 3.6]{HW-hSLE}, if $\gamma$ is SLE$_\kappa$ in $\HH$ from $0$ to $\infty$, there is an associated uniformly integrable martingale $(M_t)_{t\ge 0}$ satisfying
$$M_0=|x-y|^{-2\bb}r^{\frac 2\kappa}F(r),\quad M_\infty=H^{\HH\sem \gamma}_{(x,y)},$$
where $F:=_2F_1(\frac 4\kappa,1-\frac 4\kappa,\frac 8\kappa,\cdot)$ is a hypergeometric function. Let $G(x)=x^{\frac 2\kappa}F(x)$. Then
$$\int \mu^{\HH}_{(0,\infty)}(d\gamma) H^{\HH\sem \gamma}_{(x,y)}=\EE[M_\infty]=M_0=H^{\HH}_{(x,y)}G(r),$$
which combined with (\ref{H1},\ref{H2}) and the conformal covariance of $H^U_{\alpha_j}$, $j=1,2$, implies that
$H^U_\alpha=G(r )H^U_{\alpha_1}H^U_{\alpha_2}\in (0,\infty) $. Thus, $Q^U_\alpha$ is finite.
Since the cross ratio of  $\alpha$ and $\sigma(\alpha)$ in $U$ agree, we get $H^U_\alpha=H^U_{\sigma(\alpha)}$.

By \cite[Proposition 6.10]{HW-hSLE}, for any simply connected domain $D$ and $\beta\in\LP^D_2$, the $2$-SLE$_\kappa$ law associated with $\beta$, denoted by $\mu^D_\beta$,  exists  uniquely and satisfies conformal invariance. This immediately implies that $\sigma(\mu^U_\alpha)=\mu^U_{\sigma(\alpha)}$ and $\mu^U_\alpha=f^{-1}(\mu^{\HH}_{((0,\infty),(x,y))})$.
That proposition and \cite[Proposition 3.6]{HW-hSLE} together imply that the first marginal of $\mu^{\HH}_{((0,\infty),(x,y))}$ is absolutely continuous w.r.t.\ $\mu^{\HH}_{(0,\infty)}$ with Radon-Nikodym derivative $M_\infty/M_0=H^{\HH\sem\gamma}_{(x,y)}/(H^{\HH}_{(x,y)} G(r))$. Thus, by (\ref{conditional}),
$$\mu^{\HH}_{((0,\infty),(x,y))}(d(\gamma_1,\gamma_2))=H^{\HH\sem\gamma_1}_{(x,y)}/(H^{\HH}_{(x,y)} G(r)) \cdot \mu^{\HH}_{(0,\infty)}(d\gamma_1)\rotimes \mu^{\HH\sem \gamma_1}_{(x,y)}(d\gamma_2).$$
By this formula and the conformal invariance of SLE$_\kappa$ and $2$-SLE$_\kappa$,
$$ \mu^U_\alpha(d(\eta_1,\eta_2))=H^{\HH\sem f(\eta_1)}_{(x,y)}/(H^{\HH}_{(x,y)} G(r)) \cdot \mu^{U}_{\alpha_1}(d\eta_1)\rotimes \mu^{U\sem \eta_1}_{\alpha_2}(d\eta_2).$$
By the construction of $Q^U_\alpha$ and (\ref{fdot}),
$$Q^U_\alpha(d(\eta_1,\eta_2))=H^U_{\alpha_1} H^{U\sem \eta_1}_{\alpha_2}\cdot \mu^U_{\alpha_1}(d\eta_1)\rotimes \mu^{U\sem \eta_1}_{\alpha_2}(d\eta_2).$$ 
Comparing the above two displayed formulas and using $H^{U\sem \eta_1}_{\alpha_2}/{H^{\HH\sem f(\eta_1)}_{(x,y)}}=H^U_{\alpha_2}/H^{\HH}_{(x,y)}$ and $H^U_\alpha=G(r )H^U_{\alpha_1}H^U_{\alpha_2}$, we get $Q^U_\alpha=H^U_\alpha\mu^U_\alpha$, which together with  $H^U_\alpha=H^U_{\sigma(\alpha)}$ and $\sigma(\mu^U_\alpha)=\mu^U_{\sigma(\alpha)}$ implies  $ \sigma (Q^U_{\alpha}) =Q^U_{\sigma(\alpha)}$.
\end{proof}

\begin{Lemma}
  For any $\alpha\in \CLP^U_N$ and $\sigma\in S_N$, $\sigma(Q^U_\alpha)=Q^U_{\sigma(\alpha)}$.
  \label{commute}
\end{Lemma}
\begin{proof}
By Lemma \ref{2-commute}, the statement holds for $N=1,2$. Suppose $N\ge 3$.
  By (\ref{cascade2}),
  $$Q^U_\alpha=Q^U_{\alpha^{N-2}}(d(\eta_1,\dots,\eta_{N-2}))\rotimes Q^{U\sem \bigcup_{j=1}^{N-2} \eta_j}_{(\alpha_{N-1},\alpha_N)}(d(\eta_{N-1},\eta_N)).$$
  By Lemma \ref{2-commute}, for $\sigma_{N}:=(N-1,N)\in S_N$, $\sigma_{N}(Q^U_\alpha)=Q^U_{\sigma_{N}(\alpha)}$. This result applied to $\alpha^k$ and $\sigma_k=(k-1,k)\in S_k$ together with Lemma \ref{2-commute} implies that $\sigma_{k}(Q^U_{\alpha^k})=Q^U_{\sigma_{k}(\alpha^k)}$, $2\le k\le N$. We view $\sigma_k$ as an element of $S_m$ for any $k\le m\le N$. By (\ref{cascade2}) and induction, we get $\sigma_k(Q^U_{\alpha^m})=Q^U_{\sigma_k(\alpha^m)}$ for $k\le m\le N$. Taking $m=N$, we get $\sigma_k(Q^U_\alpha)=Q^U_{\sigma_k(\alpha)}$. The proof now finishes since $\{\sigma_k:2\le k\le N\}$ generates the group $S_N$.
\end{proof}

\begin{Lemma}
For $\alpha\in \LP^U_N$, $Q^U_\alpha$ is an $N$-SLE$_\kappa$ measure in $U$ with link pattern $\alpha$. \label{Thm}
\end{Lemma}
\begin{proof}
Fix $j\in\N_N$. We need to show that,
\BGE Q^U_\alpha = \sigma_{[j,N]}( P_{\ha j}(Q^U_\alpha)(d(\eta_1,\dots,\ha\eta_j,\dots,\eta_N))) \rotimes    \mu^{U\sem \bigcup_{k\ne j} \eta_k}_{\alpha_j}(d\eta_j).\label{j}\EDE
That (\ref{j}) holds for $j=N$ follows  from (\ref{cascade}) and (\ref{P1linnu}). By Lemma \ref{commute}, ${\sigma}_{[j,N]}^{-1}(Q^U_\alpha)=Q^U_{\sigma_{[j,N]}^{-1}(\alpha)}$. 
Since $\sigma_{[j,N]}^{-1}(\alpha)=(\alpha_1,\dots,\alpha_{j-1},\alpha_{j+1},\dots,\alpha_N,\alpha_j)$ and (\ref{j}) holds for $j=N$,
$${\sigma}_{[j,N]}^{-1}(Q^U_\alpha)=P_{\ha N}({\sigma}_{[j,N]}^{-1}(Q^U_\alpha))(d(\eta_1,\dots,\ha\eta_j,\dots,\eta_N))\rotimes \mu^{U\sem \bigcup_{k\ne j} \eta_k}_{\alpha_j}(d\eta_j)$$
$$=P_{\ha j}(Q^U_\alpha) (d(\eta_1,\dots,\ha\eta_j,\dots,\eta_N))\rotimes \mu^{U\sem \bigcup_{k\ne j} \eta_k}_{\alpha_j}(d\eta_j),$$
where the last ``$=$'' follows from (\ref{P-N-j}).
This immediately implies (\ref{j}).
\end{proof}


Combining Lemmas \ref{Thm}, \ref{Lem-trivial} and \ref{Lem-finite}, we obtain the existence theorem.

\begin{Theorem} [Existence]
Let $\kappa\in (0,6]$. Then for any $\alpha\in \LP^U_N$, $Q^U_\alpha/|Q^U_\alpha|$ is an $N$-SLE$_\kappa$ law in $U$ with link pattern $\alpha$.
\end{Theorem}



Apparently, the existence statement holds for any $\kappa\in(6,8)$ such that $Q^U_\alpha$ is finite.

\section{Uniqueness}
In this section we  prove the uniqueness of $N$-SLE$_\kappa$ law in a simply connected domain with any given link pattern.
We will use the Markov chain theory in \cite{Markov}.
 By  Proposition 10.1.1 and Theorem 10.4.9 of the book, if a Markov chain $\Phi$  is  {\it irreducible} and admits an invariant probability measure $\mu$, then all ($\sigma$-finite) invariant measures of $\Phi$ are some constant times $\mu$. By  Section 4.2.1 of the book, $\Phi$ is called ($\phi$-)irreducible if there is a nontrivial $\sigma$-finite measure $\phi$ such that if  $A$ satisfies $\phi(A)>0$, then for any $x$ in the space, $L(x,A)>0$, which is equivalent to  $\sum_{m=1}^\infty P^m(x,A)>0$. Such $\phi$ is called an irreducibility measure for $\Phi$. 

 Fix a simply connected domain $D$ with locally connected boundary and $\alpha\in\LP^D_N$. For $j\in\N_N$, let $X_j$ be the space of unparametrized curves in $\lin D$ which connect points of $\alpha_j$. 
 Let $\ulin X$ be the space of $(\eta_1,\dots,\eta_N)\in \prod_{j=1}^N X_j$ such that for any $j\in\N_N$, $\alpha_j\in \LP^{D\sem \bigcup_{k\ne j}\eta_k}_1$. For any $\ulin A\subset \ulin X$, $\ulin \eta\in \ulin X$, and $j\in\N_N$, let $\pi^{\ulin \eta}_j(\ulin A)$   denote the set of $\gamma\in X_j$ such that if the $j$-th component of $\ulin \eta$ were replaced by $\gamma$, then the modified $\ulin \eta$ would lie in $\ulin A$. 

 Define a Markov chain $\Phi$ on $\ulin X$, whose transition kernel is $P=\frac 1N\sum_{j=1}^N P_j$, where
   $$P_j(\ulin \eta,\ulin A)=P_j((\eta_1,\dots,\eta_N),\ulin A)= \mu^{D\sem \bigcup_{k\ne { j}} \eta_k}_{\alpha_j}(\pi^{\ulin \eta}_j(\ulin A)).$$
   In plain words, given  $\Phi_0=\ulin\eta^0$, the next step $\Phi_1=\ulin\eta^1$ is given by $\eta^1_k=\eta^0_k$ for $k\ne \bf j$ and $\eta^1_{\bf j}=\gamma$, where $\bf j$ is chosen from $\N_N$ uniformly randomly, and $\gamma$ is sampled according to the SLE$_\kappa$ law $\mu^{D\sem \bigcup_{k\ne {\bf j}} \eta_k^0}_{\alpha_{\bf j}}$.  
   It is obvious that if $\mu$ is an $N$-SLE$_\kappa$ measure in $D$, then it is also an invariant measure for $\Phi$. 

\begin{Lemma}
  $\Phi$ is irreducible.
\end{Lemma}
\begin{proof}
%
  We use the notion of tubes in the proof of Lemma \ref{Lem-trivial}, and  find mutually disjoint tubes ${T}_j$, $j\in\N_N$, in $D$ such that $T_j$ connects $\alpha_j$. Let ${\cal T}_j=\{\eta\in X_j:\eta\subset T_j\}$, and $\phi_j=\ind_{ {{T}_j}}\mu^D_{\alpha_j}$, i.e., the law of the chordal SLE$_\kappa$  in $D$ connecting $\alpha_j$ restricted to the event that the curve stays in ${T}_j$.   By Proposition \ref{Prop-abs}, each $\phi_j$ is a nontrivial finite measure on $X_j$. Let $\phi=\prod_{j=1}^N \phi_j$. We are going to show that $\phi$ is an irreducibility measure for $\Phi$.

  Suppose $\ulin{\cal A}\subset\ulin X$ satisfies that $\phi(\ulin{\cal A})>0$. Let $\ulin \eta=(\eta_1,\dots,\eta_N)\in \ulin X$. We need to show that $P^m(\ulin \eta,\ulin{\cal A})>0$ for some $m\in\N$. It suffices to construct measurable $\ulin{\cal A}^0,\dots,\ulin{\cal A}^{m}\subset \ulin X$ with $\ulin{\cal A}^0=\{\ulin \eta\}$ and $\ulin{\cal A}^m=\ulin{\cal A}$ such that for any $\ulin \gamma\in \ulin{\cal A}^{s-1}$, $1\le s\le m$, $P(\ulin \gamma,\ulin{\cal A}^{s})>0$, which holds if
  there is $j_s\in\N_N$ such that $P_{j_s}(\ulin \gamma,\ulin{\cal A}^{s})>0$.

  First, suppose $\ulin{\cal A}= \ulin{\cal T}:=\prod_{j=1}^N {{\cal T}_j}$. By reordering the components of $\alpha$ and using Lemma \ref{commute}, we may assume that for any $1\le j\le N$, there is an open boundary arc $I_j$ of $D$ with endpoints $\alpha_j$ such that ${\cal I}_j:=\{k\in\N_N:\alpha_k\subset I_j\}\subset \N_{j-1}$.  We have the following fact.
  \begin{itemize}
    \item [(F)] Let $r\in\N_N$. Let $A,B$ be closed subset of $\lin D$ such that $A\cup  {I_r}$ is connected, $A\cap ({\pa D\sem I_r}) =\emptyset$, $B\cup  ({\pa D\sem I_r})$ is connected, $B\cap \lin{I_r}=\emptyset$, and $A\cap B=\emptyset$. Then there is a tube connecting $\alpha_{r}$ in $D$, which is disjoint from $A\cup B$.
 %
  \end{itemize}

Applying (F) to $r=1$, $A=\emptyset$ and $B=\bigcup_{j=2}^N(\eta_j\cup T_j)$, we construct a tube $S_1$ connecting $\alpha_1$, which is disjoint from all $\eta_j,T_j$, $2\le j\le N$.  Suppose for some $m\le N-1$, we have constructed mutually disjoint tubes $S_j$ connecting $\alpha_j$ for $1\le j\le m$, which are   disjoint from all $\eta_k,T_k$ for $k\ge m+1$. Applying (F) to $r=m+1$, $A=\bigcup_{j\in {\cal I}_{m+1}} S_j$  and $B=\bigcup_{j\in \N_m\sem  {\cal I}_{m+1}} S_j \cup \bigcup_{j=m+2}^N (\eta_j\cup T_j)$,   we construct a tube $S_{m+1}$ connecting $\alpha_{m+1}$, which is disjoint from $S_j$ for $1\le j\le m$ and $\alpha_j\cup T_j$ for $  m+2\le j\le N$. By induction we get mutually disjoint tubes $S_j$ connecting $\alpha_j$, $1\le j\le N$, such that $S_j$ is disjoint from $\eta_k$ and $T_k$ whenever $j<k$. Let ${\cal S}_j=\{\eta\in X_j:\eta\subset S_j\}$.  For $1\le s\le N$, we define $\ulin {\cal A}^s=\prod_{j=1}^{s} {\cal S}_j\times \prod_{j={s+1}}^N \{\eta_j\}$. Note that $\ulin{\cal A}^0=\{\ulin \eta\}$ and $\ulin {\cal A}^N=\prod_{j=1}^N {\cal S}_j=:\ulin{\cal S}$. Let $s\in\N_N$ and $\ulin \gamma=(\gamma_1,\dots,\gamma_N)\in {\cal A}^{s-1}$. Then $\gamma_j\subset S_j$ for $1\le j\le s$ and $\gamma_j=\eta_j$ for $s+1\le j\le N$. Since $S_s$ is disjoint from $S_j$, $j<s$, and $\eta_j$, $j>s$, it is a tube in $D\sem \bigcup_{j\ne s} \gamma_j$ connecting $\alpha_j$. Thus, by Proposition \ref{Prop-abs},
 $P_{ s}(\ulin \gamma,\ulin{\cal A}^{s}) =\mu^{D\sem \bigcup_{j\ne s} \gamma_j}_{\alpha_s}(\pi^{\ulin \gamma}_s(\ulin {\cal A}^s))=\mu^{D\sem \bigcup_{j\ne s} \gamma_j}_{\alpha_s}({\cal S}_s)>0$.

For $1\le r\le N$, we let $\ulin {\cal A}^{N+r}=\prod_{j=1}^{N-r} {\cal S}_j\times \prod_{j=N+1-r}^{N} {\cal T}_j$. Note that $\ulin {\cal A}^{2N}= \ulin{\cal T}$. Let $r\in \N_N$ and $\ulin \gamma=(\gamma_1,\dots,\gamma_N)\in \ulin{\cal A}^{N+r-1}$. Since $T_{N+1-r}$ is disjoint from $T_j$, which contains $\gamma_j$, for $j>N+1-r$ and $S_j$, which also contains $\gamma_j$, for $j<N+1-r$, it is a tube in $D\sem \bigcup_{j\ne s} \gamma_j$ connecting $\alpha_{N+1-r}$. Thus, by Proposition \ref{Prop-abs},
$$P_{N+1-r} (\ulin \gamma,\ulin{\cal A}^{N+r})=\mu^{D\sem \bigcup_{j\ne N+1-r} \gamma_j}_{\alpha_{N+1-r}} (\pi^{\ulin \gamma}_{N+1-r}(\ulin {\cal A}^{N+r}))=\mu^{D\sem \bigcup_{j\ne N+1-r} \gamma_j}_{\alpha_{N+1-r}}({\cal T}_{N+1-r})>0.$$
So the sequence $\ulin{\cal A}^s$, $0\le s\le 2N$, satisfy the desired properties for $\ulin{\cal A}=\ulin{\cal T}$.

Now we consider the general case. Since $\phi$ is supported by $\ulin{\cal T}$, we may assume $\ulin{\cal A}\subset \ulin{\cal T}$. It suffices to further construct  $\ulin{\cal A}^{2N+t}$, $1\le t\le N$, with $\ulin{\cal A}^{3N}=\ulin{\cal A}$, such that for any $\ulin \gamma\in \ulin{\cal A}^{2N+t-1}$, $1\le t\le N$,   $P_{t}(\ulin \gamma,\ulin{\cal A}^{2N+t})>0$. We define these $\ulin{\cal A}^{2N+t}$'s backward inductively such that $\ulin{\cal A}^{3N}=\ulin{\cal A}$ and
\BGE \ulin{\cal A}^{2N+t-1}=\{\ulin \gamma\in \ulin{\cal T}:\phi_t(\pi^{\ulin \gamma}_t(\ulin {\cal A}^{2N+t}))>0\}\label{back-induc}\EDE for
$2\le t\le N$. Since $\phi=\prod_{j=1}^N\phi_j$ and $\phi(\ulin{\cal A}^{3N})>0$, by Fubini theorem and backward induction, each $\ulin{\cal A}^{2N+t}$ has the form $ {\cal C}_t\times \prod_{j=t+1}^N {\cal T}_t$ for some $ {\cal C}_t\subset\prod_{j=1}^t {\cal T}_j$ with $\prod_{j=1}^t \phi_j ({\cal C}_t)>0$. From $\ulin{\cal A}^{2N+1}={\cal C}_1\times \prod_{j=2}^N {\cal T}_j$ and $\phi_1({\cal C})>0$ we see that (\ref{back-induc}) also holds for $t=1$.

Let $1\le t\le N$ and $\ulin \gamma\in \ulin{\cal A}^{2N+t-1}$. By (\ref{back-induc}), $\mu^D_{\alpha_t}(\pi^{\ulin \gamma}_t(\ulin {\cal A}^{2N+t}))=\phi_t(\pi^{\ulin \gamma}_t(\ulin {\cal A}^{2N+t}))>0$. Since $T_t$ is a tube in $D\sem \cup_{j\ne t} \gamma_j$ connecting $\alpha_t$,  by  Proposition \ref{Prop-abs},
$$P_t(\ulin \gamma,\ulin {\cal A}^{2N+t})=\mu^{D\sem \bigcup_{j\ne t} \gamma_j}_{\alpha_t}(\pi^{\ulin \gamma}_t(\ulin {\cal A}^{2N+t}))=\mu^{D_t(\ulin \gamma)}_{\alpha_t}(\pi^{\ulin \gamma}_t(\ulin {\cal A}^{2N+t}))>0.$$
Thus, $\ulin{\cal A}^s$, $0\le s\le 3N$, satisfy the desired properties for the general $\ulin{\cal A}$.
\end{proof}

The irreducible Markov chain theory then implies the uniqueness theorem.

\begin{Theorem} [Uniqueness]
  Let $\kappa\in (0,8)$. Let $D$ be a simply connected domain with locally connected boundary and $\alpha\in \LP^D_N$. If there exists an $N$-SLE$_\kappa$ law $\mu$ in $D$ with link pattern $\alpha$, then  $Q^D_\alpha$ is finite, and its normalization $Q^D_\alpha/|Q^D_\alpha|$ equals $\mu$.
\end{Theorem}

\section{Partition function}
For $\alpha\in\LP^U_N$, we call $H^U_\alpha:=|Q^U_\alpha|$ a multiple SLE$_\kappa$ partition function. When $N=1,2$, the function is finite and has a closed formula. For the rest of the section we assume  $\kappa\le 6$, in which case $H^U_\alpha$ is known to be finite by Lemma \ref{Lem-finite}.  By (\ref{covariance},\ref{cascade}), $H^U_\alpha$ satisfies the following conformal covariance. Suppose $f$ is a conformal map on $U$, which extends to neighborhoods of points in
 $\alpha $.  Then
\BGE  H^U_\alpha=\prod_{x\in\alpha } |f'(x)|^{\bb} H^{f(U)}_{f(\alpha)}.\label{covariance-Q}\EDE
Thus, it suffices to consider the case that $U=\HH$. In that case, $\LP_N:=\LP^{\HH}_N$ is an open subset of $\R^{2N}$. 
By a slight abuse of notation, we also call the function $H(\alpha):=H^{\HH}_\alpha$  defined on $\LP_N$ a multiple ($N$-)SLE$_\kappa$ partition function.

Our definition follows the theme of \cite{KL-multiple}, where the partition function is defined as the total mass of a measure. It is different from the multiple SLE partition function introduced in \cite{Julien-Euler,Julien}, in which the function is assumed to satisfy a system of $2N$ second-order linear PDE (cf.\ \cite[Formula (1.2)]{KP-multiple}), which guarantees the existence of local multiple SLE.

A local multiple SLE$_\kappa$ consists of $2N$ SLE$_\kappa$-type curves emanating from the points of $\alpha$ and stopped before leaving neighborhoods of these points, which satisfy some commutation relation. These curves may or may not be extended to form  an $N$-SLE$_\kappa$ associated with   $\alpha$. Some authors call the usual multiple SLE {\it global multiple SLE}, and call a partition function {\it multiple SLE pure partition function} whenever it generates a local multiple SLE that is a part of a global multiple SLE with a fixed link pattern.


By (\ref{covariance-Q}), $H$ satisfies the conformal covariance for multiple SLE partition function (cf.\ \cite[Formula (1.3)]{KP-multiple}). We will show that $H$ satisfies the asymptotics for pure partition function (cf.\ \cite[Formula (1.3)]{HW-hSLE}) in Lemma \ref{Lem-ASY}.

For $N\ge 2$ and $j\in \N_N$, let $P_{\ha j}$ denote the map $\ulin x\mapsto(x_1,\dots,\ha{x_{2j-1}},\ha{x_{2j}},\dots,x_{2N})$ from $\R^{2N}$ to $\R^{2(N-1)}$. This is consistent with our definition of $P_{\ha j}:\LP_N\to\LP_{N-1}$.

Let $D$ be a simply connected domain and $\alpha\in\LP^D_N$. For two prime ends $a\ne b$ of $D$, which are different from points in $\alpha$, we define
\BGE I^{D;(a,b)}_\alpha=\int \mu^D_\alpha(d\eta) H^{U\sem\eta}_\alpha.\label{I}\EDE
The conformal covariance of $H^U_\alpha$ implies the following conformal covariance of $I^{D;(a,b)}_{\alpha}$. If $f$ satisfies the conditions of the $f$ in (\ref{covariance-Q}), then
\BGE I^{U;(a,b)}_\alpha=\prod_{x\in\alpha} |f'(x)|^{\bb} H^{f(U);(f(a),f(b))}_{f(\alpha)}.\label{covariance-I}\EDE
By (\ref{cascade}) and Lemma \ref{commute}, if $N\ge 2$, then for any $j\in\N_N$,
\BGE H^D_{\alpha}=H^D_{\alpha_j} I^{D;\alpha_j}_{P_{\ha j}(\alpha)}.\label{HI}\EDE


\begin{Lemma}
  Let $N\ge 2$ and $j\in\N_N$. Let $\ulin y\in\R^{2N}$ be such that $y_{2j}=y_{2j-1}\ne y_k$ for any $k\in\N_{2N}\sem \{2j,2j-1\}$, and $P_{\ha j}(\ulin y)\in\LP_{N-1}$. Then
\BGE \lim_{\LP_N\ni \ulin x\to \ulin y} |x_{2j}-x_{2j-1}|^{2\bb} H(\ulin x)\to H(P_{\ha j}(\ulin y)).\label{ASY}\EDE
\label{Lem-ASY}
\end{Lemma}
\begin{proof}
By  (\ref{I},\ref{HI}),
\BGE |x_{2j}-x_{2j-1}|^{2\bb} H(\ulin x)=\int \mu^{\HH}_{(x_{2j},x_{2j-1})} H^{\HH\sem \eta}_{P_{\ha j}(\ulin x)}.\label{ASY2}\EDE
As $\ulin x\to \ulin y$, $x_{2j},x_{2j-1}\to y_{2j}$, so the SLE$_\kappa$ curve in $\HH$ from $x_{2j}$ to $x_{2j-1}$ shrinks to the single point $y_{2j}$, which implies that the integrand in (\ref{ASY2}) tends to the RHS of (\ref{ASY}). So we get (\ref{ASY}) using dominated convergence theorem and monotonicity of $H$.
\end{proof}


\begin{Lemma}
  For any $N\in\N$, $H$ is continuous on ${\LP}_{N}$. 
\end{Lemma}
\begin{proof}
  The continuity of $H$ when $N=1$ is obvious because $H(a,b)=|a-b|^{-2\bb}$. Suppose $H$ is continuous on $\LP_k$ for any $k\le N$. We now show that $H$ is continuous on ${\LP}_{N+1}$.
  Suppose $\ulin x^n\to\ulin x^0$ in $\LP_{N+1}$. We need to show that $H(\ulin x^n)\to H(\ulin x^0)$. We may and will assume that $x^n_1=0$ and $x^n_2=1$ for all $n\in\N\cup\{0\}$ by applying (\ref{covariance}) to linear functions.

  By  (\ref{I},\ref{HI}),
  $$H(\ulin x^n)= \int \mu^{\HH}_{(0,1)}(d\eta) H^{\HH\sem \eta}_{P_{\ha 1}(\ulin x^n)},\quad n\in\N\cup\{0\}.$$
  For any fixed $\eta$, $H^{\HH\sem \eta}_{P_{\ha 1}(\ulin x^n)}$ equals the product of the partition functions in each component of $\HH\sem \eta$ whose boundary contains points of $\ulin x^n$.
  By applying the continuity of $H$ on $\LP_k$ for $k\le N$ and the conformal covariance of $H$ to each component of $\HH\sem\eta$, we get
  $H^{\HH\sem\eta}_{P_{\ha 1}(\ulin x^n)}\to H^{\HH\sem\eta}_{P_{\ha 1}(\ulin x^0)}$. By monotonicity, $H^{\HH\sem\eta}_{P_{\ha 1}(\ulin x^n)}\le  H(P_{\ha 1}(\ulin x^n))\to H(P_{\ha 1}(\ulin x^0))$. So we get $H(\ulin x^n)\to H(\ulin x^0)$ by dominated convergence theorem. The proof is completed by induction.
\end{proof}

\begin{Theorem}
    Let $\ulin u=(u_1,\dots,u_{2N})\in \LP_N$, $N\ge 2$. Let $B_t$ be a standard Brownian motion. Let $\{a,b\}=\{2j_0-1,2j_0\}$ for some $j_0\in\N_N$. Let $W_t$ and $V_t$  solve the system of SDE/ODE
  \BGE \left\{\begin{array}{ll}
  dW_t=\frac{\kappa-6}{W_t-V_t}\,dt+\sqrt\kappa dB_t, & W_0=u_a;\\
  dV_t=\frac{2}{V_t-W_t}\,dt, & V_0=u_b.
  \end{array}
  \right.
  \label{WV'}
  \EDE
  Let $g_t$ be the chordal Loewner maps driven by $W_t$, and let $\tau$ be the first time that $g_t(u_j)$ is not defined for some $j\in\N_{2N}\sem\{a\}$. Let $U^j_t=g_t(u_j)$ for $j\ne a$, and $U^a_t=W_t$. Note that $U^b_t=V_t$. Then
  \BGE  M_t:=\Big(\prod_{j\not\in \{a,b\}}  |g_t'(u_j)|\Big)^{\bb} H(W_t,V_t)^{-1} H(U^1_t,\dots,U^{2N}_t), \quad 0\le t<\tau,\label{til-M'}\EDE
  is a bounded local martingale.  \label{Lem-martingale'}
\end{Theorem}
\begin{proof}
  If $N=1$, $M_t$ is constant $1$, and the conclusion is obvious. Suppose $N\ge 2$. By symmetry we may assume that $a=1$ and $b=2$. Let $K_t$ and $\gamma(t)$, $0\le t<\tau$, be the chordal Loewner hulls and curve driven by $W_t$.    By SLE coordinate changes (\cite{SW}), $\gamma$ is a part of a chordal SLE$_\kappa$ curve, say $\eta$, in  $\HH$ from $x_1$ to $x_2$. Let $Z=H^{\HH\sem \eta}_{(u_3,\dots,u_{2N})}$. Let $T$ be any stopping time. Let $\eta_T$ and $\eta^T$ be respectively the parts of $\eta$ before and after $T$. By DMP of SLE$_\kappa$, conditionally on $\F_T$ and the event $E_T:=\{T<\tau\}$, $\eta^T$ is an SLE$_\kappa$ in $\HH\sem K_T$ from $\eta(T)$ to $u_2$. Since
$Z=H^{(\HH\sem \eta_T)\sem \eta^T}_{(u_3,\dots,u_{2N})}=H^{(\HH\sem K_T)\sem \eta^T}_{(u_3,\dots,u_{2N})}$ on $E_T$, we get
$\EE[Z|\F_T,E_T]=I^{\HH\sem K_t;(\eta(T),u_2)}_{(u_3,\dots,u_{2N})}$.
 Since $g_T$ maps $\HH\sem K_T$ conformally onto $\HH$, and sends $\eta(T)$ and $u_2$  respectively to $W_t$ and $V_T$, by (\ref{covariance-I},\ref{HI}), on the event $E_T$, $$M_T=\Big(\prod_{j\not\in \{1,2\}}  |g_t'(u_j)|\Big)^{\bb}I^{\HH;(W_t,V_t)}_{(U^3_t,\dots,U^{2N}_t)}=I^{\HH\sem K_t;(\eta(T),u_2)}_{(u_3,\dots,u_{2N})}=\EE[Z|\F_T]$$ on the event $E_T=\{T<\tau\}$. Since this holds for any stopping time $T$, and $Z\le H(u_3,\dots,u_{2N})$ by monotonicity, we get the conclusion.
\end{proof}

Once we know that $H$ is smooth,  It\^o's calculus together with (\ref{Loewner},\ref{Loewner'}) implies that $H$ satisfies the PDE for multiple SLE partition function (cf.\ \cite[Formula (1.2)]{KP-multiple}):
\BGE {\cal D}_r H:= \Big[\frac\kappa 2 \pa_r^2   +\sum_{j\ne r}  \Big(\frac 2{x_j} \pa_j -\frac {2\bb}{x_j^2}\Big)\Big] H=0,\quad 1\le r\le N,
\label{PDE-H}
\EDE

For $\kappa\in (0,6]$, if $(\eta_1,\dots,\eta_N)$ is an $N$-SLE$_\kappa$  in $D$ with link pattern $\alpha$, then for every $j\in\N_N$, the law of $\eta_j$ is absolutely continuous w.r.t.\ $\mu^{D}_{\alpha_j}$ with Radon-Nikodym derivative $H^{D\sem \eta}_{P_{\ha j}(\alpha)}/H^D_\alpha$. To see this for $j=1$, we use (\ref{cascade}) to conclude that
$Q^D_\alpha=Q^D_{\alpha_1}(d\eta_1)\otimes Q^{D\sem \eta_1}_{P_{\ha 1}(\alpha)}(d(\eta_2,\dots,\eta_N))$,
and then apply (\ref{P1mu*}). The statement for other $j$ follows from symmetry. Unlike the case $\kappa\in (0,4]$, if $\kappa>4$, for any $j\ne k\in\N_N$, the joint law of $(\eta_j, \eta_k)$  is not absolutely continuous w.r.t.\ $ \mu^{D}_{\alpha_j}\times \mu^D_{\alpha_k}$.

The partition function may also be used to describe the law of local multiple SLE. Let $(\eta_1,\dots,\eta_N)$ be multiple SLE$_\kappa$ in $U$ with link pattern $\alpha$. Suppose $\alpha_j=(a_j,b_j)$, and let $U^a_j,U^b_j$ be neighborhoods of $a_j,b_j$ in $D$, $j\in\N$,   mutually disjoint closures. Let $\eta^a_j$ (resp.\ $\eta^b_j$) be the part of $\eta_j$ started from $a_j$ (resp.\ $b_j$) and stopped on exiting  $U^a_j$ (resp.\ $U^b_j$). Then the joint law of
$(\eta^a_1,\dots,\eta^a_N;\eta^b_1,\dots,\eta^b_N)$
is absolutely continuous w.r.t.\ the product of the laws of SLE$_\kappa$ in $D$ from $a_j$ to $b_j$ stopped when exiting $U^a_j$ for $1\le j\le N$ and the laws of SLE$_\kappa$ in $D$ from $b_j$ to $a_j$ stopped when exiting $U^b_j$ for $1\le j\le N$, and the Radon-Nikodym derivative could be expressed in terms of multiple SLE$_\kappa$ partition function and Brownian loop measure.

\end{document}